\documentclass[11pt]{article}

\usepackage[T1]{fontenc}
\usepackage{amsmath}
\usepackage{amssymb}
\usepackage{amscd}
\usepackage{amsbsy}
\usepackage{amsfonts}
\usepackage{amsthm}
\usepackage{bbm}
\usepackage{mathtools}
\usepackage{xcolor}
\usepackage{comment}

\usepackage{devanagari}
\usepackage{verbatim}

\input xy
\xyoption{all}

\newtheorem{thm}{Theorem}[section]
\newtheorem{lem}[thm]{Lemma}
\newtheorem{prop}[thm]{Proposition}
\newtheorem{cor}[thm]{Corollary}
\newtheorem{rem}[thm]{Remark}
\newtheorem{dfn}[thm]{Definition}

\newtheorem{ques}[thm]{Question}

\DeclareMathOperator{\DM}{\mathrm{DM_{gm}}}
\DeclareMathOperator{\DMI}{\mathrm{DM^I_{gm}}}
\DeclareMathOperator{\DMIiso}{\mathrm{DM^{I,iso}_{gm}}}
\DeclareMathOperator{\DMiso}{\mathrm{DM^{iso}_{gm}}}

\DeclareMathOperator{\Z}{{\mathbb Z}}

\DeclareMathOperator{\Fi}{{\mathbb F}}

\DeclareMathOperator{\vs}{\mathrm{vec}}

\DeclareMathOperator{\spec}{\mathrm{Spec}}

\DeclareMathOperator{\cof}{\mathrm{cofib}}

\DeclareMathOperator{\ce}{\mathrm{\check{C}}}
\DeclareMathOperator{\colim}{\mathrm{colim}}

\DeclareMathOperator{\Hom}{\mathrm{Hom}}

\DeclareMathOperator{\chowi}{\mathrm{Chow^I}}
\DeclareMathOperator{\chowiiso}{\mathrm{Chow^{I,iso}}}

\DeclareMathOperator{\chowiso}{\mathrm{Chow^{iso}}}
\DeclareMathOperator{\chow}{\mathrm{Chow}}
\DeclareMathOperator{\End}{\mathrm{End}}
\DeclareMathOperator{\Ch}{\mathrm{Ch}}
\DeclareMathOperator{\tr}{\mathrm{Tr}}
\DeclareMathOperator{\tate}{\mathrm{Tate}}

\title{\textsc{Isotropic localizations of mixed initial motives}}
\author{Charles De Clercq, Fabio Tanania}
\date{}

\begin{document}

\maketitle

 \begin{abstract}
In this paper, we define the triangulated category of initial motives, which contains, for example, all motives of projective homogeneous varieties for semisimple algebraic groups of inner type. By analyzing its Chow weight structure, we prove that the family of isotropic localizations indexed by finitely generated extensions of the base field is conservative and Pic-injective on initial motives. This result provides an extension, in the mixed setting, of the work of Quéguiner-Mathieu and the first author on higher Tate traces of Chow motives.
 	\end{abstract}
 
\section{Introduction}
The tensor triangulated category of geometric motives $\DM(k,\Fi_p)$ \cite{V3}, where $k$ is a field and $p$ is a prime different from the characteristic of $k$, is a rich and interesting category. It contains an incredible amount of arithmetic and geometric information about smooth schemes over $k$, and for this reason it is very hard to come up with results that describe its overall structure. A possible strategy to overcome this issue is to construct new functors from $\DM(k,\Fi_p)$ to more tractable categories, and study their properties. A desirable property, for example, is conservativity, which allows one to detect the vanishing of an object in the source category by checking the vanishing of its image in the target category. Another important property is Pic-injectivity that provides a tool to describe the Picard group of $\DM(k,\Fi_p)$ by studying the Picard group of the target. 

In past decades, many realization functors, such as Betti, étale, de Rham etc., have been constructed on the category of motives. More recently, in a complementary direction, we have witnessed the development of a new type of realization functor, namely isotropic localizations. Isotropic motives were introduced by Vishik in \cite{V1}. The triangulated category of isotropic motives $\mathrm{DM^{iso}}(k,\mathbb{F}_p)$ is obtained by smashing $\mathrm{DM}(k,\Fi)$ with a certain idempotent, the isotropic unit, which annihilates all motives of $p$-anisotropic varieties, that is, varieties whose closed points have degree divisible by $p$. While this localization alone clearly loses some information, considering the family of all isotropic localizations indexed by finitely generated field extensions $E/k$,
$$L^{\mathrm{iso}}_E:\mathrm{DM}(k,\mathbb{F}_p)\rightarrow \mathrm{DM^{iso}}(E,\mathbb{F}_p),$$
allows us to retain more of the source category's structure.

This family provides a fundamental new tool for studying $\DM(k,\Fi_p)$. For example, it was used in \cite{V2} to provide new information on the Balmer spectrum of $\DM(k,\Fi_p)$, whose structure is remarkably intricate. Moreover, isotropic localizations were studied at the stable homotopic level, in various forms, by the second author in \cite{T1}, \cite{T2}, \cite{T3} and \cite{T4}.

At this point, one may ask whether the collection of isotropic realizations is conservative. An argument provided by Vishik in \cite[Example 2.13]{V1}, using elliptic curves without complex multiplication, provides a negative answer. This obstruction leads to a natural and interesting question:

\begin{ques}
\normalfont
On which subcategories of $\DM(k,\Fi_p)$ is the family of isotropic realizations indexed by finitely generated field extensions of the base field conservative?
\end{ques}

An example of such a subcategory, for $p=2$, is provided by Bachmann in \cite{B}, where a conservative and Pic-injective family of functors is constructed on the subcategory of $\DM(k,\Fi_2)$ generated by quadrics. Although not originally presented in terms of isotropic localizations, this construction can indeed be reformulated in this setting, as remarked by Vishik in \cite[Remark 2.11]{V1}. 

In this paper, we consider another large subcategory of $\DM(k,\Fi_p)$, for any prime $p$, on which the family of isotropic realizations $\{L^{\mathrm{iso}}_E\}_{E/k}$ restricts to a conservative family. We call this the triangulated category of initial motives $\DMI(k,\Fi_p)$, which contains, for example, all motives of projective homogeneous varieties for semisimple algebraic groups of inner type.

Pure initial motives were introduced in \cite{DQ} by Quéguiner-Mathieu and the first author. They are defined as pure shifts of upper motives of geometrically split, geometrically irreducible varieties satisfying the nilpotence principle. Upper motives were introduced previously by Karpenko in \cite{Kar} as direct summands $M$ of motives of smooth projective varieties such that $\Ch^0(M) \ncong 0$. As such, pure initial motives belong to the realm of Chow motives, forming an idempotent-complete subcategory satisfying the Krull-Schmidt principle. One of the main results of \cite{DQ} states that any initial Chow motive decomposes as the sum of a pure Tate motive, called Tate trace, and an anisotropic part. Moreover, two initial Chow motives are isomorphic if and only if they have isomorphic higher Tate traces over all finitely generated field extensions $E/k$.

In this paper, we obtain a variant of this result for the whole triangulated category of initial motives, which is simply the thick subcategory of $\DM(k,\Fi_p)$ generated by pure initial motives, with higher Tate traces replaced by isotropic localizations. More precisely, we prove the following:

\begin{thm}[Theorem \ref{conspic}]
The functor $\prod_{E/k}L^{iso}_E: \DMI(k,\Fi_p) \rightarrow \prod_{E/k}\DMIiso(E,\Fi_p)$ is conservative and Pic-injective.
\end{thm}

To prove the theorem, inspired by Bachmann's approach in \cite{B}, we use Chow weight structures à la Bondarko \cite{Bo}. This strategy allows us to investigate isotropic realization functors on initial motives by reducing the mixed case to the pure one, where all the relevant information is encoded in the higher Tate traces. Indeed, we show that, once we restrict the isotropic localizations to the heart of the Chow weight structure on $\DMI(k,\Fi)$, we get exactly the higher Tate traces on pure initial motives, as defined in \cite{DQ}.

We conclude by pointing out that, in recent work \cite{DKQ}, Karpenko, Quéguiner-Mathieu and the first author introduce Artin-Tate traces, bringing Artin motives into the picture. This makes it possible to study Chow motives of projective homogeneous varieties for a larger class of algebraic groups. Following this direction, we plan to study, in future work, isotropic localizations restricted to the corresponding categories of mixed motives, where we expect the family of functors considered here to remain conservative and Pic-injective, while taking values in isotropic Artin-Tate motives instead of isotropic Tate motives.

 \section{Preliminaries}
 
In this initial section, we fix some notations and recall some results that will be useful throughout this paper.
 
 Set $\Fi \coloneqq \Fi_p$ for some prime $p$, and let $k$ be a field of characteristic different from $p$. Denote by $\DM(k,\Fi)$ the category of geometric motives over $k$ with $\Fi$-coefficients. This is the thick subcategory of $\mathrm{DM}(k,\Fi)$ generated by $M(X)(i)$ for all $X \in \mathrm{Sm}(k)$ and $i \in \Z$.
 
 \begin{rem}\normalfont
For shifts of an object $M$ in $\DM(k,\Fi)$, we use the standard notation $M(q)[p]$. Anyways, since in this paper we mainly use pure shifts, we will denote the object $M(q)[2q]$ simply by $M\{q\}$.
 \end{rem}
 
 The category of geometric motives $\DM(k,\Fi)$ is a rigid tensor category equipped with a weight structure whose heart consists of Chow motives \cite{Bo}, that is:
 $$\DM(k,\Fi)^{w=0} \simeq \chow(k,\Fi).$$
 
 Recall that a functor $F:{\mathcal C} \rightarrow {\mathcal D}$ between categories endowed with weight structures is called $w$-exact if $F({\mathcal C}^{w\geq 0}) \subset {\mathcal D}^{w\geq 0}$ and, similarly, $F({\mathcal C}^{w\leq 0}) \subset {\mathcal D}^{w\leq 0}$. Moreover, $F$ is called $w$-conservative if it detects weights, namely $F(M) \in {\mathcal D}^{w\geq 0}$ implies $M \in {\mathcal C}^{w\geq 0}$ and, similarly, $F(M) \in {\mathcal D}^{w\leq 0}$ implies $M \in {\mathcal C}^{w\leq 0}$, for any object $M$ in $\mathcal C$. 
 
Since the weight structure on $\DM(k,\Fi)$ is bounded, that is,
$$\DM(k,\Fi) = \bigcup_{n \in \Z}\DM(k,\Fi)^{w \geq n} =\bigcup_{n \in \Z}\DM(k,\Fi)^{w \leq n},$$
the weight complex functor
$$\DM(k,\Fi) \rightarrow K^b(\chow(k,\Fi))$$
is $w$-exact and $w$-conservative \cite[Proposition 17 (3)]{B}.

Moreover, the weight structure on $\DM(k,\Fi)$ is non-degenerate \cite[Proposition 17 (2)]{B}, namely
$$\bigcap_{n \in \Z}\DM(k,\Fi)^{w \geq n} =\bigcap_{n \in \Z}\DM(k,\Fi)^{w \leq n}=0,$$
which implies that the weight complex functor is also conservative.
 
 Now, we move to recall from \cite{V1} definitions and constructions concerning isotropic motives.

\begin{dfn}
\normalfont
The isotropic idempotent $\Fi^{iso}$ is the object in $\mathrm{DM}(k,\Fi)$ defined by:
 $$\Fi^{iso} \coloneqq \colim_{X \: p\mathrm{\text{-}anis.}}\cof(M(\ce(X))\rightarrow \Fi),$$
where the colimit is taken over all $p$-anisotropic schemes $X$, namely schemes whose closed points have degree divisible by $p$, and $\ce(X)$ is the \v{C}ech nerve of $X$.
\end{dfn}
 
 Since $\Fi^{iso}$ is an idempotent monoid in $\mathrm{DM}(k,\Fi)$, we have a smashing localization functor
 $$L^{iso}:\mathrm{DM}(k,\Fi) \rightarrow \mathrm{DM}(k,\Fi)$$
 that on objects is given by $M^{iso} \coloneqq L^{iso}(M) \simeq M \otimes \Fi^{iso}$. By definition, since 
 $$M(\ce(X)) \otimes M(X) \simeq M(\ce(X) \times X) \simeq M(X),$$
 if $X$ is a $p$-anisotropic scheme, then $M(X)^{iso} \simeq 0$ (see \cite[Section 2]{T1}).
 
\begin{dfn}
\normalfont
Denote by $\mathrm{DM^{iso}}(k,\Fi)$ the essential image of $L^{iso}$. This is the localizing subcategory of $\mathrm{DM}(k,\Fi)$ whose objects are of the form $M \otimes \Fi^{iso}$ for all $M \in \mathrm{DM}(k,\Fi)$. The category $\mathrm{DM^{iso}}(k,\Fi)$ is called the category of isotropic motives. 
\end{dfn}
 
 Let $\DMiso(k,\Fi)$ be the thick subcategory of $\mathrm{DM^{iso}}(k,\Fi)$ generated by  $M(X)^{iso}(i)$ for all $X \in \mathrm{Sm}(k)$ and $i \in \Z$. By \cite[Proposition 5.7]{V2}, the category of geometric isotropic motives is also a rigid tensor category and carries a weight structure whose heart consists of isotropic Chow motives \cite[Definition 2.17]{V1}, that is: 
 $$\DMiso(k,\Fi)^{w=0} \simeq \chowiso(k,\Fi).$$
 
 By definition, the isotropic localization $L^{iso}$ restricts to geometric motives in the respective categories:
 $$L^{iso}:\DM(k,\Fi) \rightarrow \DMiso(k,\Fi)$$
 and, moreover, it is $w$-exact. Actually, this is true for the composite functors:
 $$L^{iso}_E:\DM(k,\Fi) \rightarrow \DM(E,\Fi) \rightarrow \DMiso(E,\Fi)$$
 for all field extensions $E/k$, since all pullback functors $\DM(k,\Fi) \rightarrow \DM(E,\Fi)$ are $w$-exact.
 
 \begin{rem}
 \normalfont
For any $E/k$, we denote by $M_E^{iso}$ the object $L^{iso}_E(M)$ in $\DMiso(E,\Fi)$. We drop the subscript only for the trivial extension.
 \end{rem}
 
At this point, we are ready to recall what we need from \cite{DQ} regarding initial Chow motives.
 
\begin{dfn}
\normalfont
Denote by $\chowi(k,\Fi)$ the subcategory of $\chow(k,\Fi)$ generated by pure initial motives, that is, pure shifts of upper motives of geometrically split, geometrically irreducible varieties satisfying the nilpotence principle. 
\end{dfn}

The category $\chowi(k,\Fi)$ contains, for example, all motives of projective homogeneous varieties for semi-simple algebraic groups of inner type.
 
\begin{dfn}
\normalfont
Let $\DMI(k,\Fi)$ be the thick subcategory of $\DM(k,\Fi)$ generated by pure initial motives. We call this category the triangulated category of initial motives.
\end{dfn}

Note that the weight structure on $\DM(k,\Fi)$ restricts to $\DMI(k,\Fi)$ by \cite[Lemma 18]{B} and
 $$\DMI(k,\Fi)^{w=0}\simeq \chowi(k,\Fi),$$
from which it follows, in particular, that $\DMI(k,\Fi)$ contains the thick subcategory of $\DM(k,\Fi)$ generated by motives of projective homogeneous varieties for semi-simple algebraic groups of inner type.
 
 We can perform the same constructions at the isotropic level. More precisely, define $\chowiiso(k,\Fi)$ as the subcategory of $\chowiso(k,\Fi)$ generated by $M^{iso}$ for any initial Chow motive $M$, and $\DMIiso(k,\Fi)$ as the thick subcategory of $\DMiso(k,\Fi)$ generated by $\chowiiso(k,\Fi)$. By the same arguments provided above, the weight structure of $\DMiso(k,\Fi)$ restricts to $\DMIiso(k,\Fi)$ and 
  $$\DMIiso(k,\Fi)^{w=0}\simeq \chowiiso(k,\Fi).$$
  
  \section{Main results}
  
  Now that we have recalled the key notions and properties of weight structures and isotropic localizations, we are ready to prove our main results.
  
  First, we want to establish a clear relation between isotropic motives \cite{V1} and Tate traces \cite{DQ}. Recall from \cite{DQ} that, for any $M \in \chowi(k,\Fi)$, there is a decomposition $M \cong \tr(M) \oplus M_{an}$, where $\tr(M)$ is the Tate trace of $M$, that is, a maximal Tate summand, and $M_{an}$ is its anisotropic part.
  
  Our first lemma says that the isotropic localization functor annihilates the anisotropic part of every initial Chow motive.
  
  \begin{lem}\label{isotr}
  	For any $M \in \chowi(k,\Fi)$, there is an isomorphism:
  	$$M^{iso} \cong \tr(M)^{iso}.$$
  \end{lem}
\begin{proof}
	We have to prove that $M_{an}^{iso}$ is zero. Notice that there is a decomposition:
	$$M_{an} \cong \bigoplus_{k=1}^m U_{X_k}\{i_k\}$$
	where all the $X_k$ are $p$-anisotropic, otherwise $U_{X_k} \cong \Fi\{i_k\}$ would fall in the Tate trace. Therefore,
	$$M_{an}^{iso} \cong \bigoplus_{k=1}^m U_{X_k}^{iso}\{i_k\} \cong 0$$
	since $U_{X_k}^{iso}$ is a direct summand of $M(X_k)^{iso} \cong 0$.
	\end{proof}

\begin{dfn}
\normalfont
Denote by $\tate(k,\Fi)$ the subcategory of $\chowi(k,\Fi)$ consisting of pure Tate motives. We know that $\tate(k,\Fi) \simeq \Fi{\mathrm -} \vs_*$.
\end{dfn}

Then next proposition identifies pure Tate motives with isotropic initial Chow motives.

\begin{prop}\label{equiv}
The composition
$$\tate(k,\Fi) \rightarrow \chowi(k,\Fi) \rightarrow \chowiiso(k,\Fi)$$
is an equivalence of categories.
\end{prop}

\begin{proof}
By Lemma \ref{isotr}, the composition is essentially surjective. The fully faithfulness follows from the fact that $\Ch_i(\spec(k)) \cong \Ch_i^{iso}(\spec(k))$ for all $i$.
\end{proof}

The next lemma ensures that the splitting of a Tate summand off an initial Chow motive can be checked on its isotropic localization.

\begin{lem}\label{split}
Let $M$ be an object in $\chowi(k,\Fi)$. Then, $\Fi^{iso}$ is a direct summand of $M^{iso}$ if and only if $\Fi$ is a direct summand of $M$.
\end{lem}
\begin{proof}
By Lemma \ref{isotr}, $\Fi^{iso}$ is a direct summand of $M^{iso}$ if and only if it is a direct summand of $\tr(M)^{iso}$. On the other hand, Proposition \ref{equiv} implies that $\Fi^{iso}$ is a direct summand of $\tr(M)^{iso}$ if and only if $\Fi$ is a direct summand of $\tr(M)$, and so of $M$.
\end{proof}

Once established the connection between isotropic initial motives and Tate traces, we need some technical results on the category of Chow motives that allow to lift maps from a field extension to the base field.

\begin{lem}\label{tech1}
 	Let $X$, $Y$ be $k$-varieties. For every morphism $\alpha:\Fi\{j\}\rightarrow M(X_{k(Y)})$, there is a morphism $\beta:M(Y)\{j\} \rightarrow M(X)$ such that the composition:
 	$$\Fi\{j\} \overset{c}{\longrightarrow} M(Y_{k(Y)})\{j\} \xrightarrow{\beta_{k(Y)}} M(X_{k(Y)})$$
    coincides with $\alpha$, where $c$ is the canonical morphism induced by the generic point $\iota:\spec(k(Y))\rightarrow Y_{k(Y)}$.
 \end{lem}
 \begin{proof}
 By dualizing, $\alpha$ corresponds to a morphism $\alpha^{\vee}:M(X_{k(Y)})\rightarrow \Fi\{d_X-j\}$, namely an element of $\Ch^{d_X-j}(X_{k(Y)})$. Since pullback $\Ch^{d_X-j}(X \times Y) \rightarrow \Ch^{d_X-j}(X_{k(Y)})$ induced by $\iota$ is surjective, there is an element $\beta^{\vee} \in \Ch^{d_X-j}(X \times Y)$ sent to $\alpha^{\vee}$. Dualizing back, we get a morphism $\beta:M(Y)\{j\}\rightarrow M(X)$ and $\beta_{k(Y)} \circ c = [(\mathrm{id}_{X_{k(Y)}} \times \iota)^* (\beta^{\vee})]^{\vee}=\alpha$, by definition of the composition of correspondences.
 \end{proof}

 \begin{cor}\label{tech2}Let $X$, $Y$ be geometrically split $k$-varieties satisfying nilpotence principle. For every morphism $\alpha':\Fi\{j\}\rightarrow (U_{X})_{k(Y)}$, there is a morphism $\beta':U_Y\{j\} \rightarrow U_X$ such that the composition:
 	$$\Fi\{j\} \overset{c'}{\longrightarrow} (U_{Y})_{k(Y)}\{j\} \xrightarrow{\beta'_{k(Y)}} (U_{X})_{k(Y)}$$
    coincides with $\alpha'$, where $c'$ is induced by the generic point.
 \end{cor}

\begin{proof}
Write $\pi_X$ for the projector defining the summand $U_{X}$ of $M(X)$. Since $\alpha$ takes values in $(U_{X})_{k(Y)}$ we have $(\pi_X)_{k(Y)}\circ \alpha'=\alpha'$, on the level of Chow groups. Applying Lemma \ref{tech1} to the morphism given by $(\pi_X)_{k(Y)}\circ \alpha'$ yields a morphism $\beta:M(Y)\{j\} \rightarrow M(X)$ such that $\beta_{k(Y)}\circ c = (\pi_X)_{k(Y)}\circ \alpha'$.
Denoting by $\pi_Y$ the projector defining the summand $U_Y$ of $M(Y)$, set $\beta'=\pi_X\circ \beta \circ \pi_Y$. Over $k(Y)$ we have
$$(\beta')_{k(Y)}\circ c'=(\pi_X)_{k(Y)}\circ \beta_{k(Y)}\circ c= (\pi_X)_{k(Y)}\circ \alpha'=\alpha'.$$    
\end{proof}

The next lemma, which will be used in the following proposition, is a general result about split epimorphisms of a certain type in an additive category.

\begin{lem}\label{add}
Let $f:M \oplus U \longrightarrow N \oplus U$ be a morphism in an additive category of the form:
$$f \coloneqq \begin{pmatrix}
f' & g\\
0 & \mathrm{id}
\end{pmatrix}.$$

Then, $f$ is a split epimorphism if and only if $f'$ is a split epimorphism.
\end{lem}

\begin{proof}
If $f$ is a split epimorphism, then there is $\begin{pmatrix}
	a & b\\
	c & d
\end{pmatrix}$ such that:
$$\begin{pmatrix}
f' & g\\
0 & \mathrm{id}
\end{pmatrix}\begin{pmatrix}
a & b\\
c & d
\end{pmatrix}=\begin{pmatrix}
\mathrm{id} & 0\\
0 & \mathrm{id}
\end{pmatrix}.$$

It follows that $f'a+gc=\mathrm{id}$ and $c=0$, so $f'a=\mathrm{id}$, namely $f'$ is a split epimorphism.

Viceversa, if the epimorphism $f'$ is split by $a$, then $\begin{pmatrix}
	a & -ag\\
	0 & \mathrm{id}
\end{pmatrix}$ is a section of $f$.
\end{proof}

We are ready to prove the key proposition that will allow us later on to show our main theorem. It essentially states that the property of a morphism between initial Chow motives of being a split epimorphism can be checked over the isotropic localizations for all field extensions.

 \begin{prop}\label{splitepi}
Let $f:M \rightarrow N$ be a morphism in $\chowi(k,\Fi)$. If $f$ induces split epimorphisms $f_E^{iso}:M_E^{iso} \rightarrow N_E^{iso}$ over all field extensions $E/k$, then $f$ is a split epimorphism.
 \end{prop}
\begin{proof}
We know that 
$$M \cong \bigoplus_{k=1}^m U_{X_k}\{i_k\} \quad \quad \mathrm{and} \quad \quad N \cong \bigoplus_{l=1}^n U_{Y_l}\{j_l\}$$
for some upper motives $U_{X_k}$ and $U_{Y_l}$.

We proceed by induction on $n$. For $n=1$, assume the morphism $f:M \rightarrow U_Y\{j\}$ induces split epimorphisms $f_E^{iso}:M_E^{iso} \rightarrow (U_{Y})_E^{iso}\{j\}$ over all field extensions $E/k$. For $E=k(Y)$,  $\Fi^{iso}\{j\}$ is a direct summand of $(U_{Y})_{k(Y)}^{iso}\{j\}$, so it is a direct summand of $M_{k(Y)}^{iso}$ as well. Hence, there is $1 \leq k \leq m$ such that $\Fi^{iso}\{j\}$ is a direct summand of $(U_{X_k})_{k(Y)}^{iso}\{i_k\}$, which implies, by Lemma \ref{split}, that $\Fi\{j\}$ is a direct summand of $(U_{X_k})_{k(Y)}\{i_k\}$. By Corollary \ref{tech2} there is a morphism $\beta:U_Y\{j\} \rightarrow U_{X_k}\{i_k\}$ such that the map $\alpha\in \mathrm{End}(U_Y\{j\})$ given by the composition
$$U_Y\{j\} \overset{\beta}{\longrightarrow} U_{X_k}\{i_k\} \longrightarrow U_Y\{j\},$$
with the second map induced by $f$, is given over $k(Y)$ by a morphism
$$\alpha_{k(Y)}:\Fi\{j\}\oplus \widetilde{U}\{j\} \longrightarrow \Fi\{j\}\oplus \widetilde{U}\{j\}$$
whose matrix is of the form $\begin{pmatrix}
\mathrm{id} & g\\
0 & h
\end{pmatrix}$, since $\Hom(\Fi,\widetilde{U}) \cong 0$.

As the motive $U_Y$ is indecomposable in the Krull-Schmidt category $\chowi(k\Fi)$, its endomorphism ring $\End(U_Y\{j\})$ is local. Since the correspondence $\alpha$ is not nilpotent over $k(Y)$, it is an invertible element of $\End(U_Y\{j\})$. Therefore, $U_Y\{j\}$ is a retract of $U_{X_k}\{i_k\}$ and $U_{X_k}\{i_k\} \cong U_Y\{j\}$ since $U_{X_k}\{i_k\}$ is indecomposable. This concludes the induction basis.

Now, suppose the statement holds for $n-1$. Assume $f:M \rightarrow N \cong  N' \oplus U_{Y_n}\{j_n\}$ induces split epimorphisms $M_E^{iso} \rightarrow N_E^{iso}$ over all field extensions $E/k$. In particular, $M_E^{iso} \rightarrow (U_{Y_n})_E^{iso}\{j_n\}$ is a split epimorphism for all $E/k$. Hence, by the induction basis, we have an equivalence $M \cong M' \oplus U_{Y_n}\{j_n\}$. Up to an automorphism of $M' \oplus U_{Y_n}\{j_n\}$, the map $f$ has the form $f \coloneqq \begin{pmatrix}
f' & g\\
0 & \mathrm{id}
\end{pmatrix}$. This implies, by Lemma \ref{add}, that $f': M' \rightarrow N'$ induces split epimorphisms $M_E^{'iso} \rightarrow N_E^{'iso}$ over all field extensions $E/k$. Therefore, by induction hypothesis $f'$ is a split epimorphism. We conclude that $f$ is a split epimorphism, again by Lemma \ref{add}.
\end{proof}

\begin{prop}\label{cons}
	The functor $K^b(\chowi(k,\Fi)) \rightarrow \prod_{E/k}K^b(\chowiiso(E,\Fi))$ is $w$-conservative. In particular, it is conservative.
	\end{prop}
\begin{proof}
	Since the functor 
	$$K^b(\chowi(k,\Fi)) \rightarrow \prod_{E/k}K^b(\chowiiso(E,\Fi))$$
	factors through $K^b(\prod_{E/k}\chowiiso(E,\Fi))$ it suffices to show that 
	$$K^b(\chowi(k,\Fi)) \rightarrow K^b(\prod_{E/k}\chowiiso(E,\Fi))$$
	is $w$-conservative. For this, we want to use \cite[Lemma 27]{B}. 
	
	Let $M^{\bullet}$ be a complex in $K^b(\chowi(k,\Fi))$ in non-positive weight such that $M^{\bullet,iso}_E$ is in negative weight for all field extensions $E/k$. Then, there is an equivalence $M^{\bullet} \simeq N^{\bullet}$, where $N^i \cong 0$ for all $i>0$ and $N^{\bullet,iso}_E$ is in negative weight for all field extensions $E/k$. This means that $N^{i,iso}_E\cong 0$ for $i>0$ and $N^{-1,iso}_E \rightarrow N^{0,iso}_E$ is a split epimorphism for all $E/k$. By Proposition  \ref{splitepi}, the map $N^{-1}\rightarrow N^0$ is a split epimorphism, which implies that $N^{-1} \cong N^0 \oplus L$ for some $L$, since $\chowi(k,\Fi)$ is idempotent-complete. It follows that $N^{\bullet} \simeq \widetilde{N}^{\bullet}$ where $\widetilde{N}^{i} \cong N^i$ for $i\neq 0,-1$, $\widetilde{N}^{-1}\cong L$ and $\widetilde{N}^{0}\cong 0$. Hence, $M^{\bullet} \simeq N^{\bullet} \simeq \widetilde{N}^{\bullet}$ is in negative weight, which implies the $w$-conservativity of the functor. Since the weight structure on $K^b(\chowi(k,\Fi))$ is non-degenerate, the functor is also conservative.
\end{proof}

\begin{rem}
\normalfont
   We point out that restricting the functor
$$\prod_{E/k}L^{iso}_E: \DMI(k,\Fi) \rightarrow \prod_{E/k}\DMIiso(E,\Fi)$$
to the heart of the weight structure of $\DMI(k,\Fi)$ yields a conservative functor:
$$\chowi(k,\Fi) \rightarrow \prod_{E/k}\chowiiso(E,\Fi) \simeq \prod_{E/k}\Fi{\mathrm -}\vs_{*}$$
that sends $M$ to $\{\tr(M_E)\}_{E/k}$ by Propositions \ref{isotr} and \ref{equiv}. 
\end{rem}

\begin{prop}\label{pic}
	The functor $K^b(\chowi(k,\Fi)) \rightarrow \prod_{E/k}K^b(\chowiiso(E,\Fi))$ is Pic-injective.
	\end{prop}
\begin{proof}
	Again it is enough to prove that 
	$$K^b(\chowi(k,\Fi)) \rightarrow K^b(\prod_{E/k}\chowiiso(E,\Fi))$$
	is Pic-injective.
	
	Let $M^{\bullet}$ be a complex in $K^b(\chowi(k,\Fi))$ such that $M^{\bullet,iso}_E \simeq \Fi^{iso}$ for all field extensions $E/k$. Then, by Proposition \ref{cons}, $M^{\bullet}$ is in weight 0, that is, there is an equivalence $M^{\bullet} \simeq M'$, where $M' \in \chowi(k,\Fi)$. Now, note that $M_E^{'iso} \cong \Fi^{iso}$, which is the same as $\tr(M'_E)\cong \Fi$, for all $E/k$. By \cite[Theorem 4.3]{DQ}, it follows that $M'\cong \Fi$, namely $M \simeq \Fi$, which concludes the proof. 
\end{proof}

 Now that we have established $w$-conservativity and Pic-injectivity of a certain functor between categories of complexes of initial Chow motives, we are ready to state our main theorem.
 
\begin{thm}\label{conspic}
	The functor $\DMI(k,\Fi) \rightarrow \prod_{E/k}\DMIiso(E,\Fi)$ is conservative and Pic-injective.
\end{thm}
\begin{proof}
By \cite[Corollary 3.5]{S}, since the functor $\DMI(k,\Fi) \rightarrow \prod_{E/k}\DMIiso(E,\Fi)$ is exact, there is a commutative square
$$
\xymatrix{
	\DMI(k,\Fi) \ar@{->}[r] \ar@{->}[d] &  K^b(\chowi(k,\Fi)) \ar@{->}[d] \\
	\prod_{E/k}\DMIiso(E,\Fi) \ar@{->}[r] &	\prod_{E/k}K^b(\chowiiso(E,\Fi))
}
$$
where the top horizontal functor is the weight complex functor, and so conservative and Pic-injective. By Propositions \ref{cons} and \ref{pic} the right vertical functor is also conservative and Pic-injective, from which it follows that the left vertical functor is conservative and Pic-injective as well. This completes the proof.
\end{proof}

\begin{rem}\label{inner}
\normalfont
Let $G$ be a semi-simple algebraic group of inner type over $k$ (that is, the absolute Galois group of $k$ acts trivially on the Dynkin diagram of $G$ through the $\ast$-action). The motive of any projective $G$-homogeneous variety with coefficients in $\mathbb{F}_p$ belongs to $\chowi(k,\mathbb{F}_p)$ by \cite[Theorem 3.5]{Kar}. Theorem \ref{conspic} hence applies for any prime $p$ to the thick subcategory of $\DM(k,\mathbb{F}_p)$ generated by motives of projective homogeneous varieties for semi-simple algebraic groups of inner type. 
\end{rem}

 \footnotesize{
 	
 }

\end{document}